\documentclass[11pt,reqno]{article}
\usepackage{authblk}
\usepackage{graphicx} 
\usepackage[margin=1in]{geometry}
\usepackage{xcolor}
\usepackage{natbib}
\usepackage[utf8]{inputenc}
\usepackage{amsmath, amssymb, amsthm}
\usepackage{hyperref}
\usepackage{cleveref}
\usepackage{enumitem}

\newtheorem{theorem}{Theorem}[section]
\newtheorem{lemma}[theorem]{Lemma}

\title{An Improved Last-Iterate Convergence Rate for \\ Anchored Gradient Descent Ascent}

\author{Anja Surina}
\author{Arun Suggala}
\author{George Tsoukalas}
\author{Anton Kovsharov}
\author{Sergey Shirobokov}
\author{Francisco J.\ R.\ Ruiz}
\author{Pushmeet Kohli}
\author{Swarat Chaudhuri}
\affil{Google DeepMind}

\date{}

\begin{document}

\maketitle

\begin{abstract}
We analyze the last-iterate convergence of the Anchored Gradient Descent Ascent algorithm for smooth convex-concave min-max problems. While previous work established a last-iterate rate of $\mathcal{O}(1/t^{2-2p})$ for the squared gradient norm, where $p \in (1/2, 1)$, it remained an open problem whether the improved exact $\mathcal{O}(1/t)$ rate is achievable. In this work,  we resolve this question in the affirmative. This result was discovered autonomously by an AI system capable of writing formal proofs in Lean. The Lean proof can be accessed \href{https://github.com/google-deepmind/formal-conjectures/pull/3675/commits/a13226b49fd3b897f4c409194f3bcbeb96a08515}{at this link}.
\end{abstract}

\section{Introduction}
In this work, we study the problem of finding Nash equilibria of min-max games of the form $$\min_x \max_y L(x, y),$$ where $L(\cdot, y): \mathbb{R}^n \to \mathbb{R}$ is convex and $L(x, \cdot): \mathbb{R}^m \to \mathbb{R}$ is concave.
Min-max games arise naturally across game theory, machine learning, and statistics, with practical applications ranging from training generative adversarial networks \citep{goodfellow2014generative} to adversarial robustness \citep{madry2018towards}, reinforcement learning \citep{du2017stochastic}, optimal transport \citep{pmlr-v84-alvarez-melis18a}, and fairness in classification \citep{agarwal2018reductions}.

One of the most popular algorithms for solving such problems is Gradient Descent Ascent (GDA).
At each step $t$, the algorithm updates the current value of the variables ($x_t$, $y_t$) following the steepest descent (ascent) direction, i.e.,
$$x_{t+1} = x_t - \alpha_t \nabla_x L(x_t, y_t), \ y_{t+1} = y_t + \alpha_t \nabla_y L(x_t, y_t),$$
for some step size parameter $\alpha_t > 0$, typically set to the same value for both variables $x$ and $y$.
However, with non-negative step-sizes, GDA is known to exhibit oscillatory behavior, where the updates do not lead to convergence of $x$ and $y$, even in simple bilinear games \citep{daskalakis2018training, shugart2025negative}.
To address this instability, various algorithmic modifications have been proposed, such as the Extragradient method \citep{korpelevich1976extragradient, hsieh2019convergence, cai2022tight} and the Optimistic Gradient Descent Ascent algorithm \citep{popov1980modification, golowich2020tight, cai2022accelerated, gorbunov2022last}. While effective in the absence of stochasticity, these methods can be unstable in practical machine learning settings where we only have access to noisy stochastic gradients---often requiring expensive double-sampling to maintain unbiased estimates or suffering from noise accumulation.

Recently, Anchored GDA \citep{ryu2019ode} has emerged as a compelling \emph{single-call} alternative, providing natural variance reduction and stability in stochastic settings.
Anchoring finds its heritage in Halpern iteration \citep{Halpern1967FixedPO} and modifies standard GDA by adding a term that pulls the current iterate toward a fixed ``anchor'' (often the starting point $(x_0, y_0)$):
$$x_{t+1} = x_t - \alpha_t \nabla_x L(x_t, y_t) + \beta_t(x_0 - x_t),\hspace{8pt}\  y_{t+1} = y_t + \alpha_t \nabla_y L(x_t, y_t) + \beta_t(y_0 - y_t),$$
where $\beta_t$ is an anchoring parameter.
While our work focuses on single-call Anchoring GDA, anchoring lineage also includes algorithms that require multiple gradient evaluations per iteration \citep{pmlr-v125-diakonikolas20a, yoon2021accelerated, lee2021fast, lee2021semi, JMLR:v25:22-1126}.

For a training horizon of $t$ iterations, \citet{ryu2019ode} established last-iterate convergence of Anchored GDA (see their Theorem~3), showing a rate of $\mathcal{O}(1/t^{2-2p})$ in the squared gradient norm for $p \in (1/2, 1)$, which is strictly slower than $\mathcal{O}(1/t)$. The case $p = 1/2$ is not covered by their analysis, leaving open the following question: \emph{can Anchored GDA achieve the $\mathcal{O}(1/t)$ convergence rate under the  standard assumptions, or is the rate fundamentally slower, for instance $\mathcal{O}(\log^c t / t)$ for some constant $c > 0$}? 
In this work, we use an AI system to resolve this question by proving that the algorithm does indeed achieve the $\mathcal{O}(1/t)$ convergence rate (see \Cref{sec:ai} for more information on AI contributions to the effort).

\section{Preliminaries}\label{sec:prelim}
In this section, we recapitulate the Anchored GDA algorithm \citep{ryu2019ode}.

\paragraph{Notation.} Let  $\mathcal{Z} = \mathbb{R}^n \times \mathbb{R}^m$. Let $L : \mathcal{Z} \to \mathbb{R}$ be a continuously differentiable objective function. The goal is to find a saddle point $z^* = (x^*, y^*) \in \mathcal{Z}$ such that:
\[
    L(x^*, y) \leq L(x^*, y^*) \leq L(x, y^*) \quad \forall x \in \mathbb{R}^n, y \in \mathbb{R}^m.
\]
The gradient operator $G : \mathcal{Z} \to \mathcal{Z}$ is given by:
\[
    G(z) = (\nabla_x L(z), -\nabla_y L(z)).
\]
Let $\|\cdot\|$ denote the $L_2$ norm.

\paragraph{Assumptions.} Throughout this paper, we assume the following:
\begin{enumerate}
    \item \textit{Monotonicity}: the operator $G$ satisfies $\langle G(z) - G(w), z - w \rangle \geq 0$ for all $z, w \in \mathcal{Z}$.
    
    \item \textit{Smoothness (Lipschitz continuity)}: there exists a constant $K > 0$ such that $\|G(z) - G(w)\| \leq K \|z - w\|$ for all $z, w \in \mathcal{Z}$.
    
    \item \textit{Existence of a saddle point}: there exists at least one solution $z^* \in \mathcal{Z}$ such that $G(z^*) = 0$.
\end{enumerate}

\paragraph{Algorithm:} The Anchored GDA update rule modifies the standard gradient step with an anchoring term. Starting from an initial point $z_0 \in Z$, the update rule for iterates $z_t = (x_t, y_t)$ is defined as:
\[
    z_{t+1} = z_t - \alpha_t G(z_t) + \beta_t(z_0 - z_t).
\]
To prove the convergence rate of $\mathcal{O}(1/t^{2-2p})$ for $p \in (1/2, 1)$, \citet{ryu2019ode} chose the parameter schedules as
\[
    \alpha_t = \frac{1-p}{(t + 1)^p}, \quad \beta_t = \frac{(1 - p) \gamma}{t + 1},
\]
with $\gamma \ge 2$.

\section{Main Result}

To achieve the $\mathcal{O}(1/t)$ convergence rate of Anchored GDA, we set different parameter schedules than those considered by \citet{ryu2019ode}:
\[
    \alpha_t = \frac{1}{K\sqrt{t + \gamma}}, \quad \beta_t = \frac{\gamma}{t + \gamma},
\]
with $\gamma \ge 2$. 

\begin{theorem}[Last-Iterate Convergence] \label{thm:main}
    Under the assumptions outlined in \Cref{sec:prelim} and the parameter schedules described above, for all $t \ge 1$, the squared gradient norm of the last iterate satisfies:
    \[
        \|G(z_t)\|^2 \le \frac{C}{t},
    \]
    where $C= K^2(E + \gamma D)^2$, $D = (\sqrt{12}+1)\|z_0 - z^\star\|$, $K$ is the Lipschitz constant, and $E \ge 0$ is a constant depending on $\gamma$, $\|z_0 - z^\star\|$ and $\|z_2 - z_1\|$ (see Lemma~\ref{lem:diff_contraction} for precise definition). 
\end{theorem}

\noindent The proof holds for any saddle point $z^\star$ satisfying $G(z^\star) = 0$; in particular, the result does not require uniqueness. The proof of Theorem~\ref{thm:main} proceeds in three steps:
\begin{itemize}
    \item \textit{Boundedness of iterates:} Bounding the distance of the iterates to the saddle point $z^\star$.
    \item \textit{Iterate stability:} Bounding the differences between consecutive iterates.
    \item \textit{Final convergence rate:} Establishing the $\mathcal{O}(1/t)$ rate for the squared gradient norm.
\end{itemize}

Prior convergence guarantees \citep{ryu2019ode} were derived by modeling the algorithm as a continuous-time ordinary differential equation (ODE). In contrast, our proof departs from the continuous-time analysis by directly analyzing the discrete dynamics and establishing a recurrence relation for consecutive iterate differences.

\subsection{Step 1: Bounded Iterates}
We first establish that the iterates remain within a bounded distance from the optimum $z^\star$.
\begin{lemma} \label{lem:one_step}
    The squared distance to any optimum $z^\star$ satisfies:
    \[
        \|z_{t+1} - z^\star\|^2 \le (1 - \beta_t + 1.5 \alpha_t^2 K^2) \|z_t - z^\star\|^2 + (\beta_t + 2\beta_t^2) \|z_0 - z^\star\|^2.
    \]
\end{lemma}

\begin{proof}
    Let $u = z_t - z^\star$ and $v = z_0 - z^\star$.
    From the algorithm update rule, we have $z_{t+1} - z^\star = z_{t} - \alpha_t G(z_t)+ \beta_t (z_0 - z_t) - z^* = w - \alpha_t G(z_t)$, where $w = (1-\beta_t)u + \beta_t v$. Taking the squared norm gives:
    \[
        \|z_{t+1} - z^\star\|^2 = \|w\|^2 - 2\alpha_t \langle w, G(z_t) \rangle + \alpha_t^2 \|G(z_t)\|^2.
    \]
    Expanding the cross term yields:
    \[
        - 2\alpha_t \langle w, G(z_t) \rangle = -2\alpha_t(1-\beta_t)\langle u, G(z_t) \rangle - 2\alpha_t\beta_t\langle v, G(z_t)\rangle.
    \]
    By monotonicity, $\langle u, G(z_t) \rangle \ge 0$, so the first term is non-positive. For the second term, we start with the non-negativity of the squared norm of the vector $2\beta_t v + \alpha_t G(z_t)$. Expanding $\|2\beta_t v + \alpha_t G(z_t)\|^2 \ge 0$ yields $-4\alpha_t\beta_t \langle v, G(z_t) \rangle \le 4\beta_t^2 \|v\|^2 + \alpha_t^2 \|G(z_t)\|^2$, which proves
    \[
    -2\alpha_t\beta_t\langle v, G(z_t)\rangle \le 2\beta_t^2\|v\|^2 + 0.5\alpha_t^2\|G(z_t)\|^2.
    \] 
Finally, using the convexity of the squared norm we can expand $\|w\|^2 \le (1-\beta_t)\|u\|^2 + \beta_t\|v\|^2$. Bounding $\|G(z_t)\|^2 \le K^2\|u\|^2$ using Lipschitz continuity of $G$, and summing the terms completes the proof.
\end{proof}

\begin{lemma} \label{lem:bounded_iterates}
    For all $t \ge 0$, and any optimum $z^\star$, the sequence satisfies $$\|z_t - z^\star\|^2 \le 12 \|z_0 - z^\star\|^2.$$ Consequently, there exists a constant $D = (\sqrt{12} + 1)\|z_0 - z^\star\| \ge 0$ such that $\|z_t - z_0\| \le D$ for all $t$.
\end{lemma}

\begin{proof}
    We proceed by induction on $t$. Assume $\|z_t - z^\star\|^2 \le 12 \|z_0 - z^\star\|^2$. Substituting this into Lemma \ref{lem:one_step}, we require the multiplier of $\|z_0 - z^\star\|^2$ to be less than or equal to $12$:
    \[
        12(1 - \beta_t + 1.5 \alpha_t^2 K^2) + (\beta_t + 2\beta_t^2) \le 12 \implies -11\beta_t + 18\alpha_t^2 K^2 + 2\beta_t^2 \le 0.
    \]
    Substituting the step size schedules $\alpha_t^2 K^2 = \frac{1}{t+\gamma}$ and $\beta_t = \frac{\gamma}{t+\gamma}$, we get:
    \[
        \frac{-11\gamma}{t+\gamma} + \frac{18}{t+\gamma} + \frac{2\gamma^2}{(t+\gamma)^2} \le 0 \implies -11\gamma(t+\gamma) + 18(t+\gamma) + 2\gamma^2 \le 0.
    \]
    Given our assumption $\gamma \ge 2$, inequality $18t+18\gamma \leq 11 \gamma t + 9 \gamma^2$ holds for any $t \ge 0$ . The bound on the distance to the anchor $\|z_t - z_0\|$ follows from the triangle inequality: $\|z_t - z_0\| \le \|z_t - z^\star\| + \|z^\star - z_0\|$.
\end{proof}

\subsection{Step 2: Bounded Iterate Differences}
Next, we bound the magnitude of consecutive iterate differences $d_t = z_{t+1} - z_t$.

\begin{lemma}\label{lem:diff_expansion}
    The difference between consecutive iterates can be written exactly as:
    \[
        z_{t+2} - z_{t+1} = A \cdot (z_{t+1} - z_t) - \alpha_{t+1}(G(z_{t+1}) - G(z_t)) + E_{\mathrm{err}} \cdot (z_0 - z_t),
    \]
    where $A = 1 - \beta_{t+1} - \frac{\alpha_t - \alpha_{t+1}}{\alpha_t}$ and $E_{\mathrm{err}} = \frac{\alpha_t - \alpha_{t+1}}{\alpha_t}\beta_t + \beta_{t+1} - \beta_t$.
\end{lemma}

\begin{proof}

    From the algorithm's update rule we have:
    \[
        z_{t+2} = z_{t+1} - \alpha_{t+1} G(z_{t+1}) + \beta_{t+1} (z_0 - z_{t+1}).
    \]
    Rearranging and introducing the gradient difference yields
    \[
        z_{t+2} - z_{t+1} = - \alpha_{t+1} (G(z_{t+1}) - G(z_t)) - \alpha_{t+1} G(z_t) + \beta_{t+1} (z_0 - z_{t+1}).
    \]
    Substituting $G(z_t)$ in the expression with $\frac{1}{\alpha_t} (z_t - z_{t+1}) + \frac{\beta_t}{\alpha_t} (z_0 - z_t)$ derived from the update rule $z_{t+1} = z_t - \alpha_t G(z_t) + \beta_t (z_0 - z_t)$, and rewriting $z_0 - z_{t+1} = (z_0 - z_t) - (z_{t+1} - z_t)$, we obtain:
    \begin{align*}
        z_{t+2} - z_{t+1} &= - \alpha_{t+1} (G(z_{t+1}) - G(z_t)) + \frac{\alpha_{t+1}}{\alpha_t} (z_{t+1} - z_t) - \frac{\alpha_{t+1} \beta_t}{\alpha_t} (z_0 - z_t) \\
        &\quad + \beta_{t+1} (z_0 - z_t) - \beta_{t+1} (z_{t+1} - z_t).
    \end{align*}
    Grouping the terms by $(z_{t+1} - z_t)$ and $(z_0 - z_t)$ gives:
    \[
        z_{t+2} - z_{t+1} = \left( \frac{\alpha_{t+1}}{\alpha_t} - \beta_{t+1} \right) (z_{t+1} - z_t) - \alpha_{t+1} (G(z_{t+1}) - G(z_t)) + \left( \beta_{t+1} - \frac{\alpha_{t+1}}{\alpha_t} \beta_t \right) (z_0 - z_t).
    \]
    Using the identity $\frac{\alpha_{t+1}}{\alpha_t} = 1 - \frac{\alpha_t - \alpha_{t+1}}{\alpha_t}$, the coefficient of $(z_{t+1} - z_t)$ evaluates exactly to $A$:
    \[
        \frac{\alpha_{t+1}}{\alpha_t} - \beta_{t+1} = 1 - \beta_{t+1} - \frac{\alpha_t - \alpha_{t+1}}{\alpha_t} = A.
    \]
    Applying the same identity to the coefficient of $(z_0 - z_t)$ gives $E_{\mathrm{err}}$:
    \[
        \beta_{t+1} - \left( 1 - \frac{\alpha_t - \alpha_{t+1}}{\alpha_t} \right) \beta_t = \frac{\alpha_t - \alpha_{t+1}}{\alpha_t} \beta_t + \beta_{t+1} - \beta_t = E_{\mathrm{err}}.
    \]
\end{proof}

\begin{lemma} \label{lem:diff_contraction}
    There exists a constant $E \ge 0$ such that for all $t \ge 1$:
    \[
        \|z_{t+1} - z_t\| \le \frac{E}{t+\gamma}.
    \]
\end{lemma}

\begin{proof}
    Let $d_t = z_{t+1} - z_t$. By Lemma \ref{lem:diff_expansion} and the triangle inequality, we have:
    \[
        \|d_{t+1}\| \le \|A d_t - \alpha_{t+1}(G(z_{t+1}) - G(z_t))\| + |E_{\mathrm{err}}| \|z_0 - z_t\|.
    \]
    Expanding the squared norm of the first term yields:
    \[
        \|A d_t - \alpha_{t+1}(G(z_{t+1}) - G(z_t))\|^2 = A^2 \|d_t\|^2 - 2A\alpha_{t+1}\langle G(z_{t+1}) - G(z_t), d_t \rangle + \alpha_{t+1}^2 \|G(z_{t+1}) - G(z_t)\|^2.
    \]
    By $A \geq 0$ and the monotonicity of G, the cross-term is non-positive. By the Lipschitz continuity of $G$, the squared gradient difference is bounded by $K^2 \|d_t\|^2$. This yields:
    \[
        \|A d_t - \alpha_{t+1}(G(z_{t+1}) - G(z_t))\|^2 \le (A^2 + \alpha_{t+1}^2 K^2) \|d_t\|^2.
    \]
    Taking the square root gives a contraction factor of $\sqrt{A^2 + \alpha_{t+1}^2 K^2}$. Substituting the parameter schedules, we can bound the contraction factor and the error coefficient for $t\geq1$ and $\gamma \geq 2$:
    \[
        \sqrt{A^2 + \alpha_{t+1}^2 K^2} \le 1 - \frac{1.15}{t+1+\gamma}, \quad \text{and} \quad |E_{\mathrm{err}}| \le \frac{\gamma}{(t+\gamma)^2}.
    \]
    In Appendix~\ref{sec:asymptotic_ae}, we provide a asymptotic argument for why the above two inequalities hold. We point the reader to the Lean proof for a detailed derivation of the non-asymptotic bounds. 
    Since $\|z_0 - z_t\| \le D$ by Lemma \ref{lem:bounded_iterates}, we obtain the recurrence relation:
    \[
        \|d_{t+1}\| \le \left(1 - \frac{1.15}{t+1+\gamma}\right) \|d_t\| + \frac{\gamma D}{(t+\gamma)^2}.
    \]
    By induction, this sequence implies the existence of a constant $E \ge 0$ such that $\|d_t\| \le \frac{E}{t+\gamma}$ for all $t \ge 1$. To establish this, let $E = \max(\|d_1\|(1+\gamma), 20\gamma D)$. 
    For the base case $t=1$, the condition is satisfied directly by $E \ge \|d_1\|(1+\gamma)$.
    For the inductive step, assume $\|d_t\| \le \frac{E}{t+\gamma}$. Substituting this into the recurrence relation gives:
    \begin{align*}
        \|d_{t+1} \| &\le \left(1 - \frac{1.15}{t+1+\gamma}\right) \frac{E}{t+\gamma} + \frac{\gamma D}{(t+\gamma)^2} \\
        &= \frac{t+\gamma - 0.15}{t+1+\gamma} \cdot \frac{E}{t+\gamma} + \frac{\gamma D}{(t+\gamma)^2}.
    \end{align*}
    To complete the induction, we must show $\|d_{t+1}\| \le \frac{E}{t+1+\gamma}$. We require:
    \begin{align*}
        \frac{t+\gamma - 0.15}{t+\gamma} E + \frac{\gamma D (t+1+\gamma)}{(t+\gamma)^2} &\le E, \\
        \left(1 - \frac{0.15}{t+\gamma}\right) E + \frac{\gamma D}{t+\gamma}\left(1 + \frac{1}{t+\gamma}\right) &\le E.
    \end{align*}
    Subtracting $\left(1 - \frac{0.15}{t+\gamma}\right) E$ and multiplying by $t+\gamma$ on both sides gives:
    \begin{align*}
        \gamma D \left(1 + \frac{1}{t+\gamma}\right) &\le 0.15 E.
    \end{align*}
    For $t \ge 1$ and $\gamma \ge 2$, we have $t+\gamma \ge 3$, which implies $1 + \frac{1}{t+\gamma} \le \frac{4}{3}$. Thus, setting $E \ge 20 \gamma D$ is sufficient to show this inequality holds, completing the proof.
\end{proof}

\subsection{Step 3: Last-Iterate Convergence}
Finally, we tie the bounded differences back to the gradient norm to establish the $\mathcal{O}(1/t)$ rate.

\begin{proof}[Proof of Theorem \ref{thm:main}]
    By algebraically rearranging the algorithm's update rule $z_{t+1} = z_t - \alpha_t G(z_t) + \beta_t(z_0 - z_t)$, we can isolate the operator evaluated at the last iterate:
    \[
        \alpha_t G(z_t) = (z_t - z_{t+1}) + \beta_t(z_0 - z_t).
    \]
    Taking the norm and dividing by $\alpha_t$ yields:
    \[
        \|G(z_t)\| \le \frac{1}{\alpha_t} \|z_{t+1} - z_t\| + \frac{\beta_t}{\alpha_t} \|z_0 - z_t\|.
    \]
    We now substitute our explicit parameter schedules $\alpha_t = \frac{1}{K\sqrt{t+\gamma}}$ and $\beta_t = \frac{\gamma}{t+\gamma}$:
    \[
        \|G(z_t)\| \le K\sqrt{t+\gamma} \|z_{t+1} - z_t\| + \frac{K\gamma}{\sqrt{t+\gamma}} \|z_0 - z_t\|.
    \]
    From Lemma \ref{lem:diff_contraction}, we have $\|z_{t+1} - z_t\| \le \frac{E}{t+\gamma}$. From Lemma \ref{lem:bounded_iterates}, we have $\|z_0 - z_t\| \le D$. Substituting these bounds into the equation gives:
    \[
        \|G(z_t)\| \le \frac{KE}{\sqrt{t+\gamma}} + \frac{K\gamma D}{\sqrt{t+\gamma}} = \frac{KE + K\gamma D}{\sqrt{t+\gamma}}.
    \]
    Squaring both sides completes the proof:
    \[
        \|G(z_t)\|^2 \le \frac{(KE + K\gamma D)^2}{t+\gamma}.
    \]
\end{proof}

\section{Note on AI Contribution and Formal Verification}\label{sec:ai}
The proof presented in this manuscript was discovered by an AI agent for formal mathematics developed at Google DeepMind. 
We will release the details of the agent, as well as the specific way it was deployed in the present domain, in a later paper. We present a natural language version of the Lean proof for improved readability. Gemini 3.1 Pro was used as an assistant in the presentation and informalization.

\bibliographystyle{apalike}
\bibliography{references}

@article{ryu2019ode,
  title={ODE Analysis of Stochastic Gradient Methods with Optimism and Anchoring for Minimax Problems},
  author={Ryu, Ernest K. and Yuan, Kun and Yin, Wotao},
  journal={arXiv preprint arXiv:1905.10899},
  year={2019}
}

@inproceedings{yoon2021accelerated,
  title={Accelerated algorithms for smooth convex-concave minimax problems with O (1/k\^{} 2) rate on squared gradient norm},
  author={Yoon, TaeHo and Ryu, Ernest K},
  booktitle={International conference on machine learning},
  pages={12098--12109},
  year={2021},
  organization={PMLR}
}

@article{cai2022tight,
  title={Tight Last-Iterate Convergence of the Extragradient and the Optimistic Gradient Descent-Ascent Algorithm for Constrained Monotone Variational Inequalities},
  author={Cai, Yang and Oikonomou, Argyris and Zheng, Weiqiang},
  journal={arXiv preprint arXiv:2204.09228},
  year={2022}
}

@article{golowich2020tight,
  title={Tight last-iterate convergence rates for no-regret learning in multi-player games},
  author={Golowich, Noah and Pattathil, Sarath and Daskalakis, Constantinos},
  journal={arXiv preprint arXiv:2010.13724},
  year={2020}
}

@article{cai2022accelerated,
  title={Accelerated Single-Call Methods for Constrained Min-Max Optimization},
  author={Cai, Yang and Zheng, Weiqiang},
  journal={arXiv preprint arXiv:2210.03096},
  year={2022}
}

@inproceedings{goodfellow2014generative,
  author    = {Goodfellow, Ian and Pouget-Abadie, Jean and Mirza, Mehdi and Xu, Bing and Warde-Farley, David and Ozair, Sherjil and Courville, Aaron and Bengio, Yoshua},
  booktitle = {Advances in Neural Information Processing Systems (NeurIPS)},
  title     = {Generative Adversarial Nets},
  year      = {2014}
}

@inproceedings{madry2018towards,
  title={Towards Deep Learning Models Resistant to Adversarial Attacks},
  author={Madry, Aleksander and Makelov, Aleksandar and Schmidt, Ludwig and Tsipras, Dimitris and Vladu, Adrian},
  booktitle={International Conference on Learning Representations (ICLR)},
  year={2018}
}

@inproceedings{daskalakis2018training,
  title={Training GANs with Optimism},
  author={Daskalakis, Constantinos and Ilyas, Andrew and Syrgkanis, Vasilis and Zeng, Haoyang},
  booktitle={International Conference on Learning Representations (ICLR)},
  year={2018}
}

@article{korpelevich1976extragradient,
  title={The extragradient method for finding saddle points and other problems},
  author={Korpelevich, G. M.},
  journal={Ekonomika i Matematicheskie Metody},
  volume={12},
  number={4},
  pages={747--756},
  year={1976}
}

@article{popov1980modification,
  title={A modification of the Arrow-Hurwicz method for search of saddle points},
  author={Popov, L. D.},
  journal={Mathematical Notes of the Academy of Sciences of the USSR},
  volume={28},
  number={5},
  pages={845--848},
  year={1980}
}

@article{Halpern1967FixedPO,
  title={Fixed points of nonexpanding maps},
  author={Benjamin Halpern},
  journal={Bulletin of the American Mathematical Society},
  year={1967},
  volume={73},
  pages={957-961},
  url={https://api.semanticscholar.org/CorpusID:120539954}
}

@article{lee2021fast,
  title={Fast extra gradient methods for smooth structured nonconvex-nonconcave minimax problems},
  author={Lee, Sucheol and Kim, Donghwan},
  journal={Advances in Neural Information Processing Systems},
  volume={34},
  pages={22588--22600},
  year={2021}
}

@InProceedings{pmlr-v125-diakonikolas20a,
  title = 	 {Halpern Iteration for Near-Optimal and Parameter-Free Monotone Inclusion and Strong Solutions to Variational Inequalities},
  author =       {Diakonikolas, Jelena},
  booktitle = 	 {Proceedings of Thirty Third Conference on Learning Theory},
  pages = 	 {1428--1451},
  year = 	 {2020},
  editor = 	 {Abernethy, Jacob and Agarwal, Shivani},
  volume = 	 {125},
  series = 	 {Proceedings of Machine Learning Research},
  month = 	 {09--12 Jul},
  publisher =    {PMLR},
}

@article{hsieh2019convergence,
  title={On the convergence of single-call stochastic extra-gradient methods},
  author={Hsieh, Yu-Guan and Iutzeler, Franck and Malick, J{\'e}r{\^o}me and Mertikopoulos, Panayotis},
  journal={Advances in Neural Information Processing Systems},
  volume={32},
  year={2019}
}

@article{gorbunov2022last,
  title={Last-iterate convergence of optimistic gradient method for monotone variational inequalities},
  author={Gorbunov, Eduard and Taylor, Adrien and Gidel, Gauthier},
  journal={Advances in neural information processing systems},
  volume={35},
  pages={21858--21870},
  year={2022}
}

@article{lee2021semi,
  title={Semi-anchored multi-step gradient descent ascent method for structured nonconvex-nonconcave composite minimax problems},
  author={Lee, Sucheol and Kim, Donghwan},
  journal={arXiv preprint arXiv:2105.15042},
  year={2021}
}

@article{JMLR:v25:22-1126,
  author  = {Lesi Chen and Luo Luo},
  title   = {Near-Optimal Algorithms for Making the Gradient Small in Stochastic Minimax Optimization},
  journal = {Journal of Machine Learning Research},
  year    = {2024},
  volume  = {25},
  number  = {387},
  pages   = {1--44},
  url     = {http://jmlr.org/papers/v25/22-1126.html}
}

@article{shugart2025negative,
  title={Negative stepsizes make gradient-descent-ascent converge},
  author={Shugart, Henry and Altschuler, Jason M},
  journal={arXiv preprint arXiv:2505.01423},
  year={2025}
}

@inproceedings{agarwal2018reductions,
  title={A reductions approach to fair classification},
  author={Agarwal, Alekh and Beygelzimer, Alina and Dud{\'\i}k, Miroslav and Langford, John and Wallach, Hanna},
  booktitle={International conference on machine learning},
  pages={60--69},
  year={2018},
  organization={PMLR}
}

@inproceedings{du2017stochastic,
  title={Stochastic variance reduction methods for policy evaluation},
  author={Du, Simon S and Chen, Jianshu and Li, Lihong and Xiao, Lin and Zhou, Dengyong},
  booktitle={International conference on machine learning},
  pages={1049--1058},
  year={2017},
  organization={PMLR}
}

@InProceedings{pmlr-v84-alvarez-melis18a,
  title = 	 {Structured Optimal Transport},
  author = 	 {Alvarez-Melis, David and Jaakkola, Tommi and Jegelka, Stefanie},
  booktitle = 	 {Proceedings of the Twenty-First International Conference on Artificial Intelligence and Statistics},
  pages = 	 {1771--1780},
  year = 	 {2018},
  editor = 	 {Storkey, Amos and Perez-Cruz, Fernando},
  volume = 	 {84},
  series = 	 {Proceedings of Machine Learning Research},
  month = 	 {09--11 Apr},
  publisher =    {PMLR},
  url = 	 {https://proceedings.mlr.press/v84/alvarez-melis18a.html},
}

\appendix
\section{Asymptotic Bounds for $A$ and $|E_{\mathrm{err}}|$}
\label{sec:asymptotic_ae}
 The asymptotic behavior of the terms $\sqrt{A^2 + \alpha_{t+1}^2 K^2}$ and $|E_{\mathrm{err}}|$ can be characterized as $t \to \infty$. Let us expand the terms using Taylor series:
\begin{align*}
\frac{\alpha_{t+1}}{\alpha_t} &= \sqrt{\frac{t+\gamma}{t+1+\gamma}} = 1 - \frac{1}{2(t+\gamma)} + \mathcal{O}\left(\frac{1}{t^2}\right), \\
\beta_{t+1} &= \frac{\gamma}{t+1+\gamma} = \frac{\gamma}{t+\gamma} + \mathcal{O}\left(\frac{1}{t^2}\right).
\end{align*}
Using $A = \frac{\alpha_{t+1}}{\alpha_t} - \beta_{t+1}$, we get $A \approx 1 - \frac{\gamma + 1/2}{t+\gamma}$. Squaring this and adding $\alpha_{t+1}^2 K^2 = \frac{1}{t+1+\gamma} \approx \frac{1}{t+\gamma}$ yields:
\begin{align*}
A^2 + \alpha_{t+1}^2 K^2 &\approx 1 - \frac{2\gamma + 1}{t+\gamma} + \frac{1}{t+\gamma} = 1 - \frac{2\gamma}{t+\gamma}.
\end{align*}
Taking the square root and using the Taylor expansion gives a contraction factor of $1 - \frac{\gamma}{t+\gamma}$. Because $\gamma \ge 2$, this is less than the  bound $1-\frac{1.15}{t + \gamma + 1}$ used in our proof. 
\\

\noindent For the error term $E_{\mathrm{err}} = (1 - \frac{\alpha_{t+1}}{\alpha_t})\beta_t + \beta_{t+1} - \beta_t$, we can expand each part separately. First, using the Taylor expansion for the ratio, we have
\begin{align*}
1 - \frac{\alpha_{t+1}}{\alpha_t} &= \frac{1}{2(t+\gamma)} - \frac{3}{8(t+\gamma)^2} + \mathcal{O}\left(\frac{1}{t^3}\right).
\end{align*}
Multiplying by $\beta_t = \frac{\gamma}{t+\gamma}$ gives:
\begin{align*}
\left(1 - \frac{\alpha_{t+1}}{\alpha_t}\right)\beta_t &= \frac{\gamma}{2(t+\gamma)^2} - \frac{3\gamma}{8(t+\gamma)^3} + \mathcal{O}\left(\frac{1}{t^4}\right).
\end{align*}
Next, for the difference $\beta_{t+1} - \beta_t$, we have:
\begin{align*}
\beta_{t+1} - \beta_t = \frac{-\gamma}{(t+\gamma)(t+1+\gamma)} = -\frac{\gamma}{(t+\gamma)^2} + \frac{\gamma}{(t+\gamma)^3} + \mathcal{O}\left(\frac{1}{t^4}\right).
\end{align*}
Summing these two parts yields:
\begin{align*}
E_{\mathrm{err}} = -\frac{\gamma}{2(t+\gamma)^2} + \frac{5\gamma}{8(t+\gamma)^3} + \mathcal{O}\left(\frac{1}{t^4}\right),
\end{align*}
confirming that $|E_{\mathrm{err}}| \approx \frac{\gamma}{2t^2} = \mathcal{O}(1/t^2)$.

\end{document}